\documentclass[a4paper, 12pt]{amsart}
\usepackage[top=40truemm,bottom=40truemm,left=35truemm,right=35truemm]{geometry}
\usepackage{amssymb,amsmath}

\newcommand{\C}{\mathbb{C}}

\newcommand{\Z}{\mathbb{Z}}
\newcommand{\BP}{\mathbb{P}}
\newcommand{\Q}{\mathbb{Q}}

\newcommand{\red}{{\rm red}}
\newcommand{\Bk}{\operatorname{Bk}}

\newcommand{\Supp}{\operatorname{Supp}}

\newcommand{\Pic}{\operatorname{Pic}}
\newcommand{\ch}{\operatorname{char}}
\newcommand{\id}{\operatorname{id}}
\newcommand{\mult}{\operatorname{mult}}
\newcommand{\SE}{\mathcal{E}}
\newcommand{\SO}{\mathcal{O}}
\newcommand{\SL}{\mathcal{L}}
\newcommand{\ST}{\mathcal{T}}
\newcommand{\ol}[1]{\overline{#1}}

\newcommand{\lkd}[1]{\ol{\kappa}({#1})}
\newtheorem{thm}{Theorem}[section]
\newtheorem{prop}[thm]{Proposition}
\newtheorem{lem}[thm]{Lemma}

\newtheorem{defn}[thm]{Definition}

\newtheorem{remark}[thm]{Remark}

\begin{document}
\title[Curves on irrational ruled surfaces]{Curves on irrational ruled surfaces whose complements are of non-general type}
\author{Hideo Kojima}
\address[H. Kojima]{Department of Mathematics, Faculty of Science, Niigata University, 8050 Ikarashininocho, Nishi-ku, Niigata 950-2181, Japan}
\email{kojima@math.sc.niigata-u.ac.jp}
\date{27 June 2026}
\thanks{The author is supported by JSPS KAKENHI Grant Number JP 24K06684.}
\begin{abstract}
Let $B$ be a curve on an irrational ruled surface $X$. We prove that the logarithmic Kodaira dimension of $X-B$ equals the Iitaka dimension of $K_X+B$ and give a rough configuration of $B$ when the logarithmic Kodaira dimension of $X - B$ is less than two. Next, we study the logarithmic multicanonical system of $X-B$ when the logarithmic Kodaira dimension of $X - B$ equals one and prove that its logarithmic $m$-canonical system gives either a $\BP^1$-fibration or an elliptic fibration if $m \geq 12$. 
\end{abstract}
\maketitle

\setcounter{section}{0}

\section{Introduction}

We work over an algebraically closed field $k$. In this paper, an irrational ruled surface means a smooth projective (not necessarily relatively minimal) ruled surface of positive irregularity. 

This paper is a continuation of the author's previous papers \cite{K17} and \cite{K25}. In \cite{K17}, he studied irrational open algebraic surfaces of non-negative logarithmic Kodaira dimension in any characteristic. He gave a classification of strongly minimal open irrational ruled surfaces of logarithmic Kodaira dimension zero. Further, he proved that, for a smooth irrational ruled open surface, its logarithmic Kodaira dimension is non-negative if and only if its logarithmic bi-genus is positive. 
In \cite{K25}, he studied logarithmic multicanonical systems of smooth affine surfaces of logarithmic Kodaira dimension one. Let $S$ be a smooth affine surface of logarithmic Kodaira dimension one and let $(V,D)$ be a pair of a smooth projective surface $V$ and a simple normal crossing divisor $D$ on $V$ such that $V - D \cong S$. Then, he proved that, for $m \geq 8$, the complete linear system $|\lfloor m(K_V+D)^{+} \rfloor|$ gives a $\BP^1$-fibration from $V$ onto a smooth projective curve, where $(K_V+D)^{+}$ is the nef part of the Zariski decomposition of $K_V+D$. Furthermore, as seen from \cite[Cases 2--4 in  \S 3]{K25}, we know that so does $|\lfloor m(K_V+D)^{+} \rfloor|$ for $m \geq 6$ (resp.\ $m \geq 2 + 3/(h^1(\SO_V) - 1)$ if $h^1(\SO_V) = 1$ (resp.\ $h^1(\SO_V) \geq 2$). 

Let $X$ be an irrational ruled surface and let $B$ be a reduced curve on $X$. In this paper, we first study the configuration of $B$ when the logarithmic Kodaira dimension of $X-B$ $\leq 1$. We prove the following theorem. Here, $\lkd{X-B}$ and $\kappa(X,K_X+B)$ denote respectively the logarithmic Kodaira dimension of $X-B$ and the Iitaka dimension of $K_X + B$.

\begin{thm}
Let $X$ be an irrational ruled surface with a ruling {\rm (}i.e., a $\BP^1$-fibration{\rm )} $\varphi: X \to T$ onto a smooth projective curve $T$ of genus $h^1(\SO_X)$ and let $B$ be a reduced curve on $X$. Then the following assertions hold true.
\begin{itemize}
\item[(1)]
$\lkd{X-B} = \kappa(X, K_X + B)$.
\item[(2)]
Suppose that  $h^1(\SO_X) \geq 2$ and $\lkd{X-B} \leq 1$. Then $B$ can be extended to a curve $B' (\supseteqq B)$ such that $\lkd{X-B'} = \lkd{X-B}$ and there exists a birational morphism $g: X \to \Sigma= \BP_T(\SE)$ onto a $\BP^1$-bundle over $T$, which is a composite of blowing-downs of curves in $B'$, such that $g(B')$ is one of the followings, where $\pi$ is the ruling on $\Sigma$ induced from $\varphi${\rm :}
\begin{itemize}
\item[(i)]
points on $\Sigma${\rm ;} then $\lkd{X-B} = -\infty$,
\item[(ii)]
$\ell \geq 1$ fibers of $\pi${\rm ;} then $\lkd{X-B} = -\infty$,
\item[(iii)]
one section and $\ell \geq 0$ fibers of $\pi${\rm ;} then $\lkd{X-B} = -\infty$,
\item[(iv)]
two sections and $\ell \geq 0$ fibers of $\pi${\rm ;} then $\lkd{X-B} = 1$,
\item[(v)]
one $2$-section and $\ell \geq 0$ fibers of $\pi${\rm ;} then $\lkd{X-B} = 1$.
\end{itemize}
\item[(3)]
Suppose that $h^1(\SO_X) = 1$ and $\lkd{X-B} \leq 1$. Then $B$ can be extended to a curve $B' (\supseteqq B)$ such that $\lkd{X-B} \leq \lkd{X-B'} \leq 1$ and there exists a birational morphism $g: X \to \Sigma= \BP_T(\SE)$ onto a $\BP^1$-bundle over $T$, which is a composite of blowing-downs of curves in $B'$, such that $g(B')$ is one of the followings, where $\pi$ is the ruling on $\Sigma$ induced from $\varphi${\rm :}
\begin{itemize}
\item[(i)]
points on $\Sigma${\rm ;} then $\lkd{X-B} = \lkd{X-B'} = -\infty$,
\item[(ii)]
$\ell \geq 1$ fibers of $\pi${\rm ;} then $\lkd{X-B} = \lkd{X-B'} = -\infty$,
\item[(iii)]
one section and $\ell  \geq 0$ fibers of $\pi${\rm ;} then  $\lkd{X-B} = \lkd{X-B'} = -\infty$,
\item[(iv)]
two sections and $\ell \geq 0$ fibers of $\pi${\rm ;} then $\lkd{X-B} = \lkd{X-B'} \in \{ 0,1 \}$,
\item[(v)]
one $2$-section and $\ell \geq 0$ fibers of $\pi${\rm ;} then $0 \leq \lkd{X-B} \leq \lkd{X-B'} \in \{ 0,1 \}$,
\item[(vi)]
$\ell \geq 1$ fibers of an elliptic fibration on $\Sigma$ and $3 \leq g_*(B') \cdot F$ for a fiber $F$ of $\pi${\rm ;} then $\lkd{X-B} = \lkd{X-B'} = 1$.
\end{itemize}
Furthermore, when $\ch (k) \not= 2$, the curve $B'$ can be taken such that $\lkd{X-B'} = \lkd{X-B}$.
\end{itemize}
\end{thm}

For a curve $B$ on a smooth projective surface $X$, it is clear that $\lkd{X-B} \leq \kappa(X, K_X + B)$ and the equality holds if $B$ is an SNC-divisor. In general, the equality $\lkd{X-B} = \kappa(X, K_X + B)$ does not hold true. It follows from \cite[Theorem 4]{I80} that, if $X$ is a rational surface and the curve $B$ contains at least one irrational curves, then $\kappa(X, K_X+B) = \lkd{X-B} \geq 0$. The assertion (1) of Theorem 1.1 is an irrational ruled surface version of \cite[Theorem 4]{I80}.

Suppose that $k = \C$ is the complex number field and let $B$ be a connected curve on a smooth projective surface $X$ such that the topological Euler characteristic $\chi (X - B)$ of $X-B$ is non-positive. Gurjar--Parameswaran \cite{GP95} and Veys \cite{Veys98} studied such a pair $(X,B)$ by using structure theorems for open algebraic surfaces (see, e.g., \cite{LNM857} and \cite{OAS}) and the log BMY inequality (see, e.g., \cite{KNS89}). In particular, Veys \cite{Veys98} gave structure theorems for connected curves $B$ on smooth projective ruled surfaces $X$ with $\chi(X-B) \leq 0$. We note that, for such a pair $(X,B)$, $\lkd{X-B} \leq 1$ by \cite[Proposition 2]{GP95}. Our (2) of Theorem 1.1 (resp.\ (3) of Theorem 1.1) includes \cite[Theorem 6.3]{Veys98} (resp.\ \cite[Theorem 6.6]{Veys98}). 

In Section 5, we study irrational ruled open surfaces of logarithmic Kodaira dimension one and prove the following result.

\begin{thm}
Let $X$ be an irrational ruled surface and let $B$ be a simple normal crossing divisor on $X$ such that $\lkd{X-B} = 1$. Let $(K_X + B)^{+}$ be the nef part of the Zariski decomposition of $K_X + B$. Set $g:= h^1(\SO_X)$.
\begin{itemize}
\item[(1)]
Suppose that $g \geq 2$. Then, for every integer $m \geq 2 + 3/(g-1)$, the complete linear system $| \lfloor m(K_X+B)^+ \rfloor |$ induces a $\BP^1$-fibration. Furthermore, if $\ch(k) \not= 2$, then, for every integer $m \geq 3$,  $| \lfloor m(K_X+B)^+ \rfloor |$ induces a $\BP^1$-fibration.
\item[(2)]
Suppose that $g = 1$. Then, for every integer $m \geq 12$, the complete linear system $| \lfloor m(K_X+B)^+ \rfloor |$ induces either a $\BP^1$-fibration or an elliptic fibration. Furthermore, if $\ch(k) \not= 2$, then for every integer $m \geq 8$,  $| \lfloor m(K_X+B)^+ \rfloor |$ induces either a $\BP^1$-fibration or an elliptic fibration. 
\end{itemize}
\end{thm}

In fact, in the proof of Theorem 1.2, we have more precise results especially when $|\lfloor m(K_X+D)^+ \rfloor|$ gives an fibration onto a smooth projective curve. 
\medskip

In Section 2, we recall basic notions in the theory of peeling and the construction of strongly minimal models of open algebraic surfaces. 
Let $B$ a reduced curve on an irrational ruled surface $X$. 
In Section 3, we consider the case $\lkd{X-B} = -\infty$ by using results of Miyanishi \cite{M82} and prove Theorem 1.1 in this case. In Section 4, we consider the case $\lkd{X-B} = 0$ by using results of \cite{K17} and prove Theorem 1.1 in this case. In Section 5, we study the case $\lkd{X-B} = 1$. By using the structure theorem for open algebraic surfaces of $\overline{\kappa} = 1$ (see \cite{Kaw78}, \cite{LNM857}, \cite{K13}), we study the configuration of the curve $B$, which completes the proof of Theorem 1.1. Finally, by using the logarithmic canonical bundle formula in \cite[Theorem 2.1]{K13} (which is given in \cite{Kaw78} and \cite{LNM857} when $\ch (k) = 0$) and the results on wild fibers of elliptic surfaces in \cite{KU85}, we prove Theorem 1.2.
\medskip

\noindent
{\bf Terminology.}  A reduced effective divisor $D$ is called an SNC-divisor (simple normal crossing divisor) if it has only simple normal crossings. We employ the following notations. For the definitions of $\ol{P}_m$ and $\ol{\kappa}$, see \cite{GTM76} (see also \cite{Kam80} for the definitions in any characteristic).

$K_X$: the canonical divisor on $X$.

$\lkd{S}$: the logarithmic Kodaira dimension of $S$.

$\lfloor Q \rfloor$: the integral part of a $\Q$-divisor $Q$.

$\lceil Q \rceil:= - \lfloor -Q \rfloor$: the roundup of a $\Q$-divisor $Q$.

$D_1 \sim D_2$: $D_1$ and $D_2$ are linearly equivalent.

$D_1 \equiv D_2$: $D_1$ and $D_2$ are numerically equivalent.

$\Phi_{\Gamma}$: the rational map induced from a linear system $\Gamma$.

$\mu^{-1}_*(C)$: the proper transform of $C$ by $\mu$.
\medskip

For a Cartier divisor $L$ on a smooth projective surface $X$, set $N(L) = \{ n \in \Z_{>0} \ | \ h^0(X, nL) \not= 0 \}$. Then 
we define
$$
\kappa(X,L) = \left\{\begin{array}{ll}
                         -\infty & \mbox{if}\ N(L) = \emptyset, \\
                         \max_{n \in N(L)} \{ \dim \Phi_{|nL|} (X) \} & \mbox{if}\ N(L) \not= \emptyset, 
                        \end{array}\right. 
$$
which is called the {\em Iitaka dimension} of $L$.

\section{Preliminariy results}

We recall some basic notions in the theory of peeling. For more details, see \cite[Chapter 2]{OAS} or \cite[Chapter 1]{MT84}. Let $X$ be a smooth projective surface and $B$ an SNC-divisor on $X$. We call such a pair $(X,B)$ an SNC-pair. A connected curve consisting only of irreducible components of $B$ is called a connected curve in $B$ for shortness. A connected curve $T$ in $B$ is {\em admissible}  (resp.\ {\em rational}) if there are no $(-1)$-curves in $\Supp T$ and the intersection matrix of $T$ is negative definite (resp.\ it consists only of rational curves). A connected curve $T$ in $B$ is a {\em twig} if its dual graph is a chain and $T$ meets $B-T$ in a single point at one of the end components of $T$. A connected curve $R$ (resp.\ $F$) in $B$ is a {\em rational rod} (resp.\ {\em rational fork}) if it is rational and its dual graph is a chain (resp.\ the dual graph of the exceptional curves of the minimal resolution of a Kawamata log terminal singular point and is not a chain). An admissible rational twig $T$ in $B$ is {\em maximal} if it is not extended to an admissible rational twig with more irreducible components of $B$. By a $(-2)$-rod (resp.\ a $(-2)$-fork), we mean a rational rod (resp.\ a rational fork) consisting only of $(-2)$-curves.

Let $\{ T_{\lambda} \}$ (resp.\ $\{ R_{\mu} \}$, $\{ F_{\nu} \}$) be the set of all admissible rational maximal twigs (resp.\ all admissible rational rods, all admissible rational forks). Then there exists a unique decomposition of $B$ as a sum of effective $\Q$-divisors $B= B^{\#} + \Bk(B)$ such that the following conditions are satisfied:
\begin{enumerate}
\item[(a)]
$\Supp(\Bk(B)) = (\cup_{\lambda} T_{\lambda}) \cup (\cup_{\mu} R_{\mu}) \cup (\cup_{\nu} F_{\nu})$.
\item[(b)]
$(K_X+ B^{\#}) \cdot Z = 0$ for every irreducible component $Z$ of $\Supp(\Bk(B))$.
\end{enumerate}
We call the divisor $\Bk(B)$ the {\em bark} of $B$.

\begin{lem}
With the same notations as above, each connected component of $B-\lceil B^{\#} \rceil$ is a $(-2)$-rod or a $(-2)$-fork and is a connected component of $B$. 
\end{lem}

\begin{proof}
See \cite[p.\ 94]{OAS}.
\end{proof}

\begin{defn}
{\rm An SNC-pair $(X,B)$ is {\em almost minimal} if, for every irreducible curve $C$ on $X$, either $(K_X + B^{\#}) \cdot C \geq 0$ or $(K_X+ B^{\#}) \cdot C < 0$ and the intersection matrix of $C+\Bk(B)$ is not negative definite.}
\end{defn}

It is well-known that we can construct an almost minimal model for every SNC-pair. More precisely, we have the following result. 

\begin{lem}
Let $(X,B)$ be an SNC-pair. 
\begin{itemize}
\item[(1)] 
There exists a birational morphism $\mu : X \to \tilde{X}$ onto a smooth projective surface $\tilde{X}$ such that the following four conditions {\rm (i) -- (iv)} are satisfied{\rm :}
\begin{itemize}
\item[(i)]
$\tilde{B}:= \mu_*(B)$ is an SNC-divisor.
\item[(ii)]
$\mu_* (\Bk(B)) \leq \Bk(\tilde{B})$ and $\mu_*(K_X + B^{\#}) \geq K_{\tilde{X}} + \tilde{B}^{\#}$.
\item[(iii)]
$\ol{P}_n(X-B) = \ol{P}_n(\tilde{X}-\tilde{B})$ for every integer $n \geq 1$. In particular, $\lkd{X-B} = \lkd{\tilde{X}-\tilde{B}}$.
\item[(iv)]
The pair $(\tilde{X},\tilde{B})$ is almost minimal.
\end{itemize}
\item[(2)]
Assume further that $\lkd{X-B} \geq 0$. Let $(K_X+B)^{+}$ be the nef part of the Zariski decomposition of $K_X+B$, here we note that the Zariski decomposition of $K_X+B$ exists since $K_X+B$ is then pseudo effective. Then we can take $\mu$ satisfying $\mu^*(K_{\tilde{X}} + \tilde{B}^{\#}) = (K_X+B)^{+}$. 
\end{itemize}
\end{lem}

\begin{proof}
For the proof of (1), see \cite[Theorem 2.3.11.1 (p.\ 107)]{OAS}, which is the same as \cite[Theorem 1.11]{MT84}. 
By \cite[Theorem 1.3]{Tsunoda83}, we can take $\mu$ such that  $\mu^*(K_{\tilde{X}} + \tilde{B}^{\#}) = (K_X+B)^{+}$ provided $\lkd{X-B} \geq 0$. This proves (2).
\end{proof}

\begin{lem}
Let $(X,B)$ be an almost minimal SNC-pair. Then the following assertions hold true.
\begin{itemize}
\item[(1)]
$\lkd{X-B} \geq 0$ if and only if $K_X+ B^{\#}$ is nef.
\item[(2)]
If $\lkd{X-B} \geq 0$, then  $K_X+ B^{\#}$ is semiample. Moreover, we have the following.
\begin{itemize}
\item[(2-1)]
$\lkd{X-B} = 0$ $\Longleftrightarrow$ $K_X+ B^{\#} \equiv 0$.
\item[(2-2)]
$\lkd{X-B} = 1$ $\Longleftrightarrow$ $(K_X+ B^{\#})^2 = 0$ and $K_X+ B^{\#} \not\equiv 0$.
\item[(2-3)]
$\lkd{X-B} = 2$ $\Longleftrightarrow$ $(K_X+ B^{\#})^2 > 0$.
\end{itemize}
\end{itemize}
\end{lem}

\begin{proof}
See \cite[Lemma 1.4]{K13}.
\end{proof}

Let $E$ be a $(-1)$-curve on $X$. Then $E$ is called a {\em superfluous exceptional component} of $B$ if $E \subset \Supp (\lfloor B^{\#} \rfloor)$, $E \cdot (B-E) = E \cdot (\lfloor B^{\#} \rfloor - E) = 2$ and $E$ meets two irreducible components of $\lfloor B^{\#} \rfloor - E$. Assume that $E$ is a superfluous exceptional component of $B$. Let $\mu: X \to Y$ be the contraction of $E$ and set $B_Y:= \mu_*(B)$. It is then clear that $(Y,B_Y)$ is an SNC-pair and $K_X + B^{\#} \equiv \mu^*(K_Y + B_Y^{\#})$. Further, $\ol{P}_n(X-B) = \ol{P}_n(Y-B_Y)$ for every integer $n \geq 1$. So, when we construct an almost minimal model, we assume that there exist no superfluous exceptional components.

In order to study an SNC-pair $(X,B)$ of $\lkd{X-B} \geq 0$, it is convenient to consider its strongly minimal model. We recall the following lemma. 

\begin{lem}
Let $(X,B)$ be an almost minimal SNC-pair of $\lkd{X- B} \geq 0$. Assume that there exists a $(-1)$-curve $E$ such that $E \cdot (K_X + B^{\#}) = 0$, $E \not\subset \Supp (\lfloor B^{\#} \rfloor)$ and the intersection matrix of $E+ \Bk(B)$ is negative definite. Let $\sigma: X \to Y$ be a composite of the contraction of $E$ and the contractions of all subsequently contractible components of $\Supp (\Bk(B))$. Set $B_Y:= \sigma_*(B)$. Then the following assertions hold.
\begin{itemize}
\item[(1)]
The divisor $B_Y$ is an SNC-divisor and each connected component of $\sigma(\Supp (\Bk(B)))$ is an admissible rational twig, an admissible rational rod or an admissible rational fork of $B_Y$. 

\item[(2)]
The pair $(Y,B_Y)$ is an almost minimal SNC-pair.

\item[(3)]
For every integer $n \geq 1$, $\ol{P}_n(X-B) = \ol{P}_n(Y-B_Y)$. In particular, $\lkd{Y-B_Y} = \lkd{X-B}$.
\end{itemize}
\end{lem}

\begin{proof}
All the assertions follow from \cite[(4), (6) and (7) of Lemma 2.4.4.1 (p.\ 123)]{OAS}.
\end{proof}

In Lemma 2.5, we call the pair $(V, D)$ a {\em strongly minimal model} of a given SNC-pair $(X,B)$ of $\lkd{X-B} \geq 0$. 
An SNC-pair $(V,D)$ of $\lkd{V-D} \geq 0$ is said to be {\em strongly minimal} if $(V,D)$ becomes a strongly minimal model of itself.

\section{The case $\ol{\kappa} = -\infty$}

Let $B$ be a reduced curve on an irrational ruled surface $X$. In this section, we consider the case $\lkd{X-B} = -\infty$.

\begin{lem}
Let $X$ be an irrational ruled surface and let $\varphi: X \to T$ be a $\BP^1$-fibration over a smooth projective curve $T$ of genus $h^1(\SO_X)$. Let $B$ be a reduced curve on $X$. Then the following conditions are equivalent to each other.
\begin{itemize}
\item[(1)]
For a fiber $F$ of $\varphi$, $F \cdot B \leq 1$.
\item[(2)]
$\kappa(X, K_X + B) = -\infty$.
\item[(3)]
$\lkd{X-B} = -\infty$. 
\end{itemize}
\end{lem}

\begin{proof}
In general, $\lkd{X-B} \leq \kappa(X, K_X + B)$. So, (2) implies (3). 
If (1) is true, then $F \cdot (K_X + B) = -2 + F \cdot B \leq -1$ for a fiber $F$ of $\varphi$. So $h^0(X, m(K_X+B)) = 0$ for any positive integer $m$. This proves the part  “(1) $\Longrightarrow$ (2)”.

We prove the part “(3) $\Longrightarrow$ (1)”. Assume that $\lkd{X-B} = -\infty$ and let $F$ be a fiber of $\varphi$. 
Set $s:= F \cdot B$. We may assume that $s \geq 1$. Let $f: V \to X$ be a composite of blowing-ups over points on $B$ such that $D:= f^*(B)_{\red}$ becomes an SNC-divisor. Then $\lkd{X-B} = \lkd{V-D} = \kappa(V, K_V + D)$. The map $\tilde{\varphi} := \varphi \circ f$ is  a $\BP^1$-fibration over $T$. Since every exceptional curve with respect to $f$ is a fiber component of $\tilde{\varphi}$, $\tilde{F} \cdot D = F \cdot B = s$ for a fiber $\tilde{F}$ of $\tilde{\varphi}$. Note that $|K_V + D| = \emptyset$ since $\lkd{V-D} = \kappa(V, K_V + D) = -\infty$. We infer from \cite[Lemma 1.1]{M82} that $D$ contains a unique horizontal component, say $D_1$. Further, by \cite[Lemma 1.5]{M82}, $D_1$ is a section of $\tilde{\varphi}$. Therefore, $s = 1$. 
\end{proof}

Assume that $\lkd{X-B} = -\infty$. Then we know that $B$ is an SNC-divisor. In fact, if $B \cdot F = 0$, then $B$ is contained in fibers of $\varphi$ and so $B$ is an SNC-divisor because the reduced structure of every fiber of $\varphi$ is an SNC-divisor. If $B \cdot F > 0$, then by Lemma 3.1, $B \cdot F = 1$. So $B$ consists of a section of $\varphi$ and fiber components of $\varphi$. Therefore, $B$ is an SNC-divisor. 

Let $g : X \to \BP_T(\SE)$ be a relatively minimal model of the $\BP^1$-fibration $\varphi: X \to T$, here $\BP_T(\SE)$ is a $\BP^1$-bundle over $T$ with a rank two vector bundle $\SE$ on $T$, and let $\pi$ be the $\BP^1$-fibration on $\BP_T(\SE)$ induced from $\varphi$. 
Let $E_1, \ldots, E_n$ be the exceptional curves for $g$ not contained in $B$. Set $B' = B + E_1 + \cdots + E_n$.
Then $\lkd{X-B'} = \lkd{X-B} = -\infty$.
If $g_*(B) = 0$, then $g_*(B') = 0$. 
If $B \cdot F = 0$ (resp.\ $B \cdot F = 1$) for a fiber $F$ of $\varphi$ and $g_*(B) \not= 0$,  then $g_*(B')$ consists of $\ell$ fibers of $\pi$ (resp.\  $g_*(B')$ consists of a section and $\ell$ fibers of $\pi$).
Thus, the assertions (2) and (3) of Theorem 1.1 are verified when $\lkd{X-B} = -\infty$.

\section{The case $\ol{\kappa} = 0$}

In this section, we study the case $\ol{\kappa} = 0$ and prove (2) and (3) of Theorem 1.1 in this case. 

Let $X$ be an irrational ruled surface and let $B$ be a reduced curve on $X$ such that $\lkd{X-B} = 0$. Let $\mu: \tilde{X} \to X$ be a composite of blowing-ups over points on $B$ such that $\tilde{B}:= \mu^*(B)_{\red}$ becomes an SNC-divisor. Here we may assume that $\mu$ is the shortest, i.e., the Picard number of $\tilde{X}$ is least possible. As seen from \S 2, we can construct a strongly minimal model of the SNC-pair $(\tilde{X}, \tilde{B})$. Namely, there exists a birational morphism $f: \tilde{X} \to V$ onto a smooth projective surface $V$ such that $(V, D)$, where $D = f_*(\tilde{B})$, is a strongly minimal model of $(\tilde{X}, \tilde{B})$. 
Since $X$ is irrational ruled, there exists a $\BP^1$-fibration $\varphi: X \to T$ onto a smooth projective curve $T$ of genus $h^1(\SO_X)$. Then, $\varphi$ gives a $\BP^1$-fibration from $V$ onto $T$, which we denote by $\pi$. 

Since $(V,D)$ is strongly minimal and $\lkd{V-D} = \lkd{\tilde{X} - \tilde{B}} = \lkd{X - B} = 0$, the pair $(V,D)$ is one of those in \cite[Theorem 2.1 (2)]{K17}. Namely, we have the following result.

\begin{lem}
With the same notations and assumptions as above, $V$ is an elliptic ruled surface and the following assertions hold true.
\begin{itemize}
\item[(1)]
If $V$ is relatively minimal, then either {\rm (a)} $K_V + D \sim 0$, $V = \BP_ T(\SO_T \oplus \SL)$, where $\SL \in \Pic (T)$, and $D=D_1+D_2$ is a sum of two disjoint sections $D_1$ and $D_2$ of $\pi$, or {\rm (b)} $D$ is an elliptic curve with $D \equiv -K_V$ and $V = \BP_T(\SE)$, where $\SE$ is an indecomposable vector bundle of rank two over $T$.

\item[(2)]
If $V$ is not relatively minimal, then ${\rm char} (k) = 2$ and the pair $(V,D)$ is one of the pairs constructed in \cite[Example 2.2]{K17}. In particular, the following conditions are satisfied{\rm :}
\begin{itemize}
\item[(2-1)]
Let $F_1, \ldots, F_r$ be all the singular fibers of $\pi$. Then every $F_i$ can be expressed as 
$$
F_i = 2(E_i+D^i_1+\cdots+D^i_{s_i-2})+D^i_{s_i-1}+D^i_{s_i} 
$$
$$
(s_i:= -1 + \# (F_i)_{\red} \geq 2),
$$
where $E_i$ is the unique $(-1)$-curve in $F_i$ and $D^i_j$ $(j = 1, \ldots, s_i)$ is a $(-2)$-curve and the weighted dual graph of $F_i$ looks like that of Figure {\rm 1}.
\item[(2-2)]
$D$ has a unique irreducible component $D_0$ that is a horizontal component of $\pi$. Further, $D_0$ is a $2$-section of $\pi$, $\pi|_{D_0}: D_0 \to T$ is a purely inseparable double covering, and $D_0 \cdot F_i = 2 D_0 \cdot E_i = 2$ for $i = 1, \ldots, r$. 
\item[(2-3)]
$\displaystyle D = D_0 + \sum_{i=1}^r \left(\sum_{j = 1}^{s_i} D^i_j \right)$. 
\end{itemize}
\end{itemize}
\end{lem}

\raisebox{-25mm}{
\setlength{\unitlength}{1mm}
\begin{picture}(128,30)(0,0)
\put(57,3){Figure 1}
\put(35,15){\circle{1.8}}
\put(33,18){$E_i$}
\put(33,10){$-1$}
\put(36,15){\line(1,0){13}}
\put(50,15){\circle{1.8}}
\put(48,18){$D^i_1$}
\put(48,10){$-2$}
\put(51,15){\line(1,0){4}}
\multiput(56,15)(2,0){15}{\circle*{0.2}}
\put(85,15){\line(1,0){4}}
\put(90,15){\circle{1.8}}
\put(85,19){$D^i_{s_i-2}$}
\put(88,10){$-2$}
\put(91,15){\line(3,1){14}}
\put(91,15){\line(3,-1){14}}
\put(106,20){\circle{1.8}}
\put(104,22){$-2$}
\put(110,19){$D^i_{s_i-1}$}
\put(106,10){\circle{1.8}}
\put(110,8){$D^i_{s_i}$}
\put(104,5){$-2$}
\end{picture}}

\begin{proof}
See \cite[Theorem 2.1]{K17}.
\end{proof}

\begin{remark}
In {\rm (2)} of Lemma {\rm 4.1}, since each connected component of $D-D_0$ is either a $(-2)$-rod or a $(-2)$-fork, $D^{\#} = D_0$ and so $\lkd{V-D_0} = 0$. Let $g: V \to \BP_T(\SE)$ be a relatively minimal model of $\pi: V \to T$, where $\SE$ is a vector bundle of rank two on $T$. Then $g_*(D_0)$ is smooth and the pair $(\BP_T(\SE), g_*(D_0))$ is a Frobenius pair over $T$, which is given in \cite[2.1 and 2.2]{M82}. 
\end{remark}

By (1) of Lemma 4.1, $X$ is an elliptic ruled surface. 

\begin{prop}
With the same notations and assumptions as above, $B$ is an SNC-divisor and so $\mu = \id_X$. 
\end{prop}

\begin{proof}
Let $(V,D)$ be the pair in Lemma 4.1. We consider the following cases separately. 
\medskip

\noindent
\textbf{Case 1:}  $(V,D)$ is the pair (a) in (1) of Lemma 4.1.
We use the notations in (1) of  Lemma 4.1. Let $B_i$ ($i=1,2$) be the proper transform of $D_i$ on $X$. Then $B_i$ ($i=1,2$) is a section of $\varphi$ and hence is smooth. By the constructions of $\mu$ and $f$, we see that every irreducible component of $B-(B_1+B_2)$ (if there exists) is a fiber component of $\varphi$.

Suppose to the contrary that $B$ is not an SNC-divisor. Since every irreducible component of $B$ is smooth, we know that one of the following two cases takes place.
\begin{itemize}
\item[(i)]
There exist a point $P \in B$ and an irreducible component $B_3$ of $B-(B_1+B_2)$ such that $P \in B_1 \cap B_2 \cap B_3$.
\item[(ii)]
There exists a point $P \in B_1 \cap B_2$ such that the local intersection number of $B_1$ and $B_2$ at $P$ $\geq 2$. 
\end{itemize} 
By the minimality of the Picard number $\rho(\tilde{X})$, we know that the last exceptional curve $\tilde{E}$ in the process of blowing-ups over the point $P$ satisfies the condition $\tilde{E} \cdot (\tilde{B} - \tilde{E}) \geq 3$. Since $f_*(\tilde{B}) = D$ is an SNC-divisor and $\tilde{E} \cdot (\tilde{D} - \tilde{E}) \geq 3$, we have $f_*(\tilde{E}) \not= 0$. So $f_*(\tilde{E})$ is an irreducible component of $D$ and is a rational curve. This is a contradiction. 
\medskip

\noindent
\textbf{Case 2:} $(V,D)$ is the pair (b) in (1) of Lemma 4.1 and that $\varphi|_{D}: D \to T$ is separable.
By the assumption and since $D$ and $T$ are elliptic curves, $\varphi|_{D}: D \to T$ is an unramified double covering. Let $\tilde{D}_1$ be the proper transform of $D$ on $\tilde{X}$, and $B_1 = \mu_*(\tilde{D}_1)$. Then $B_1$ is an irreducible component of $B$ and every irreducible component of $B-B_1$ is a fiber component of $\varphi$. In particular, $B-B_1$ is an SNC-divisor. Furthermore, $\varphi|_{B_1}: B_1 \to T$ is a separable morphism of degree two. 

We prove that $B_1$ is smooth. Suppose to the contrary that $B_1$ has a singular point $Q$. Let $F_Q$ be the fiber of $\varphi$ passing through $Q$. Then $F_Q \cdot B_1 = 2$ and so $\mult_Q(B_1) = 2$. Furthermore, $Q$ is not a unibranch singular point because every separable double covering between two elliptic curves is unramified. 
Since $\mu: \tilde{X} \to X$ is a sequence of blowing-ups over points on $B$ including $Q$, $\tilde{D}-\tilde{D}_1$ contains connected curve $\tilde{D}^{(0)}$ such that $\tilde{D}^{(0)} \cdot \tilde{D}_1 = 2$. Since $D = f_*(\tilde{D}) = f_*(\tilde{D}_1)$, $\tilde{D}^{(0)}$ is contracted to a point by $f$. Then $D$ is not smooth, which is a contradiction. This proves that $B_1$ is smooth. 

Since $B-B_1$ consists only of fiber components of $\varphi$ and $\varphi|_{B_1}: B_1 \to T$ is an unramified  double covering, we conclude that $B$ is an SNC-divisor. 
\medskip

\noindent
\textbf{Case 3:} Either $(V,D)$ is the pair (b) in  (1)  of Lemma 4.1 and $\varphi|_{D}: D \to T$ is not separable or $V$ is not relatively minimal.  Note that this case occurs only when $\ch(k) = 2$.
By Lemma 4.1, $D$ can be expressed as $D = D_0 + \sum_{i=1}^r (\sum_{j = 1}^{s_i} D^i_j)$, where $D_0$ is a $2$-section of $\pi$ and $\pi|_{D_0}: D_0 \to T$ is a purely inseparable double covering, $r \geq 0$ and $s_i \geq 2$ for $i = 1, \ldots, r$ (if $r > 0$). 
So $D - D_0$ is either the zero divisor or an SNC-divisor and $D_0$ is a connected component of $D$.
Let $B_0$ be the proper transform of $D_0$ on $X$. Then $B - B_0$ consists only of fiber components of $\varphi$. So $B-B_0$ is an SNC-divisor. 

Suppose that $B$ is not an SNC-divisor. By the minimality of $\rho(\tilde{X})$ (the Picard number of $\tilde{X}$), the last exceptional curve, say $\tilde{E}$, in the process of $\varphi$ is an irreducible component of $\tilde{B}$ and $\tilde{E} \cdot (\tilde{B} - \tilde{E}) \geq 2$. Furthermore, if $\tilde{E} \cdot (\tilde{B} - \tilde{E}) = 2$, then $\tilde{E}$ meets only one irreducible component of $\tilde{B}-\tilde{E}$. Since $B - B_0$ is an SNC-divisor, $\tilde{E}$ must meet $\tilde{B}_0:= \mu^{-1}_*(B_0) = f^{-1}_*(D_0)$. Since $\tilde{E}$ is a $(-1)$-curve and  $\tilde{E} \cdot (\tilde{B} - \tilde{E}) \geq 2$, we know that $f_*(\tilde{E}) \not= 0$ and $f_*(\tilde{E})$ meets $D_0$. This is a contradiction by Lemma 4.1.
\end{proof}

We prove (2) and (3) of Theorem 1.1 when $\lkd{X-B} = 0$. In fact, $X$ is then elliptic ruled and hence we do not have to prove (2) of Theorem 1.1. By Proposition 4.3, $B$ is an SNC-divisor. So, $\mu = \id$. 

Suppose that $(V,D)$ is one of  (a) and (b) in (1) of Lemma 4.1. Set $g = f$ and let $E_1, \ldots, E_n$ be the exceptional curves for $g$ not contained in $B$. Set $B' = B + E_1 + \cdots + E_n$. Then $g_*(B') = g_*(B) = C$ is (iv) with $\ell = 0$ or (v) with $\ell = 0$ in (3) of Theorem 1.1 and $\lkd{X-B'} = \lkd{X-B} = 0$. 

Suppose that $V$ is not relatively minimal. Then $\ch (k) = 2$. Let $E_1, \ldots, E_n$ be the exceptional curves for $f$ not contained in $B$. Let $h :V \to \Sigma= \BP_T(\SE)$ be a birational morphism onto a $\BP^1$-bundle $\Sigma$ over $T$. Set $g = h \circ f: X \to \Sigma$ and $B' = B + E_1 + \cdots + E_n$. Since $B' \cdot F = 2$ for a fiber $F$ of $\varphi$, $\kappa(X, K_X + B') \leq 1$. (See Lemma 5.2 for the proof.) Hence, $\lkd{X-B} = \kappa(X, K_X+B) \leq \lkd{X-B'} \leq \kappa(X, K_X+B') \leq 1$. It is clear that $g(B')$ is (v) in (3) of Theorem 1.1.

\section{The case $\ol{\kappa} = 1$}

In this section, we study the case $\ol{\kappa} = 1$. We complete the proof of Theorem 1.1 and prove Theorem 1.2.

\subsection{Preliminary results on open algebraic surfaces of $\ol{\kappa} = 1$}
Let $X$ be an irrational ruled surface and let $B$ be a reduced curve on $X$ such that $\lkd{X-B} = 1$. Let $\mu: \tilde{X} \to X$ be a composite of blowing-ups over points on $B$ such that $\tilde{B} = \mu^*(B)_{\red}$ becomes an SNC-divisor. Here we may assume that $\mu$ is minimal, i.e., the Picard number of $\tilde{X}$ is the least possible. Let $f: \tilde{X} \to V$ be a birational morphism onto a smooth projective surface $V$ such that $(V,D)$, where $D = f_*(\tilde{B})$, is a strongly minimal model of $(\tilde{X}, \tilde{B})$. 

\begin{lem}
Let the notations and assumptions be the same as above. Then, for a sufficiently large integer $n$, the complete linear system $|n(K_V+D^{\#})|$ defines a fibration $\rho: V \to \Delta$ from $V$ onto a smooth projective curve $\Delta$ such that $\rho$ is an elliptic fibration or a $\BP^1$-fibration. Moreover, let $h : V \to W$ be a birational morphism such that $\pi:= \rho \circ h^{-1}$ is a relatively minimal model of the fibration $\rho$ and set $C:= h_*(D^{\#})$. Then the following assertions hold. 
\begin{itemize}
\item[(1)]
Assume that $\pi$ is an elliptic fibration. Then $W$ is an elliptic ruled surface and $\Delta = \BP^1$. 
 Furthermore,  we have{\rm :}
\begin{itemize}
\item[(1-1)]
$C = \sum_{i=1}^r F_i$, where, for $i=1, \ldots, r$, $F_i$ is an elliptic curve with $F_i^2 = 0$ and $m_i F_i$ is a scheme-theoretic fiber of $\pi$ for some integer $m_i \geq 1$. 
\item[(1-2)]
Write $R^1 \pi_* \SO_W = \SL \oplus \ST$, where $\SL$ is a locally free $\SO_{\Delta}$-module and $\ST$ is a torsion $\SO_{\Delta}$-module. Then the divisor $K_W+C$ can be expressed as follows{\rm :}
$$
K_W +C = \pi^*(K_{\Delta}+\delta) + \sum_{r = 1}^s a_{r} E_{r} + \sum_{i=1}^j F_i, \eqno{(5.1)}
$$
\noindent
where $a_{r} E_{r}$ ranges over all multiple fibers of $\pi$ with multiplicity $m_{r}$, $0 \leq a_{r} < m_r$, $a_{r} = m_{r}  -1$ if $m_{r} E_{r}$ is not a wild fiber of $\pi$, and $\delta$ is a divisor on $\Delta$ with $\deg \delta = \chi(\SO_W) + {\rm length} \ST$. 
\end{itemize}
\item[(2)]
Assume that $\pi$ is a $\BP^1$-fibration. Then we have{\rm :}
\begin{itemize}
\item[(2-1)]
We set as $C = H + \sum_{i=1}^j d_i F_i$, where $H$ is the sum of the horizontal components of $C$ and the $F_i$'s are fibers of $\pi$. Then $H$ is an SNC-divisor and consists  of either two sections or an irreducible $2$-section of $\pi$. 
\item[(2-2)]
The divisor $K_W + C$ can be expressed as follows{\rm :}
$$
K_W + C  = \pi^*(K_{\Delta}+\delta) + \sum_{i=1}^j d_i F_i, \eqno{(5.2)}
$$
\noindent
where $\delta$ is a divisor on $\Delta$ such that $\deg \delta$ equals $H_1\cdot H_2$ {\rm (}resp.\ one half of the number of the branch points of $\pi|_H$, $1-g(\Delta)${\rm )} if $H=H_1+H_2$ with sections $H_1$ and $H_2$ {\rm (}resp.\ $H$ is irreducible and $\pi|_H$ is not purely inseparable, $H$ is irreducible and $\pi|_H$ is purely inseparable{\rm )} and 
$$ d_i =\left\{\begin{array}{ll}
                         \frac{1}{2} \left(1-\frac{1}{m_i}\right) & \mbox{if}\ \# (F_i \cap H) = 1, \\
                         1-\frac{1}{m_i} & \mbox{if}\ \# (F_i \cap H) = 2, 
                        \end{array}\right. $$
where $m_i$ is a positive integer or $+\infty$.
\end{itemize}
\end{itemize}
\end{lem}

\begin{proof}
By \cite[Theorem 2.1]{K13}, for a sufficiently large integer $n$, $|n(K_V+D^{\#})|$ defines an elliptic fibration, a quasi-elliptic fibration or a $\BP^1$-fibration, say $\rho: V \to \Delta$, onto a smooth projective curve $\Delta$. Since $V$ is an irrational ruled surface, $\rho$ cannot be a quasi-elliptic fibration. If $\rho$ is a $\BP^1$-fibration, we have (2) by \cite[Theorem 2.1 (II)]{K13}. Suppose that $\rho$ is an elliptic fibration. Then $W$ (as well as $X$, $V$) is an elliptic ruled surface and $\Delta = \BP^1$. Furthermore, every fiber of $\pi$ is a multiple of a smooth elliptic curve with self-intersection number zero. By the construction of $D^{\#}$ explained in Section 2, every irrational curve in $\Supp D$ has coefficient one in $D^{\#}$. Therefore, we obtain (1) from \cite[Theorem 2.1 (I)]{K13}. 
\end{proof}

\subsection{Proof of Theorem 1.1}

Let $X$ be an irrational ruled surface. Then we have a $\BP^1$-fibration $\varphi: X \to T$ onto a smooth projective curve $T$ of genus $h^1(\SO_X)$. Let $B$ be a reduced curve on $X$. 
We note the following result.

\begin{lem}
Let $X$, $B$ and $\varphi$ be the same as above. If $B \cdot F = 2$ for a fiber $F$ of $\varphi$, then $\kappa(X, K_X+B) = 0$ or $1$. In particular, $0 \leq \lkd{X-B} \leq \kappa(X, K_X+B) \leq 1$. 
\end{lem}

\begin{proof}
By Lemma 3.1, $0 \leq \lkd{X-B} \leq \kappa(X, K_X+B)$. Since $(K_X + B) \cdot F = 0$, $\Supp (K_X + B)$ consists only of fiber components of $\varphi$. So, $\kappa(X, K_X+B) \leq 1$. This proves the lemma. 
\end{proof}

We complete the proof of Theorem 1.1. 
\medskip

\noindent
\textbf{Case: $h^1(\SO_X) \geq 2$.}
We complete the proof of (2) of Theorem 1.1. The assertion follows from the results of Sections 3 and 4 when $\lkd{X-B} \leq 0$. 
In fact, the case $\lkd{X-B} = 0$ does not take place. We assume that $\lkd{X-B} = 1$ and use the notations in \S \S 5.1. 
Then the fibration $\rho$ in Lemma 5.1 is a $\BP^1$-fibration and $\tilde{F} \cdot D = 2$ for a fiber $\tilde{F}$ of $\rho$. So $\Delta = T$ and $\rho$ induces a $\BP^1$-fibration, say $\varphi$, on $X$. 

Let $B_0$ be the sum of horizontal components of $B$ with respect to $\varphi$. Then $B \cdot F = B_0 \cdot F = 2$ for a fiber $F$ of $\varphi$ and $B - B_0$ is contained in fibers of $\varphi$. 
Let $g: X \to \Sigma:= \BP_T(\SE)$ be a birational morphism onto a $\BP^1$-bundle $\Sigma$ over $T$ and let $E_1, \ldots, E_n$ be the exceptional curves for $g$ not contained in $B$. Set $B':= B + E_1 + \cdots + E_n$. 
Then, $1 = \lkd{X-B} \leq \lkd{X-B'} \leq \kappa(X, K_X + B') \leq 1$ by Lemma 5.2. 
So $\lkd{X-B'} = \kappa(X, K_X + B') = 1 = \lkd{X-B}$. 
It is clear that $g(B')$ is one of (iv) and (v) in (2) of Theorem 1.1. This completes the proof of (2) of Theorem 1.1.

We prove (1) of Theorem 1.1 in this case.
When $\lkd{X-B} = 1$, $\lkd{X-B} = \kappa(X, K_X + B)$ follows from the argument as in the previous paragraph.
 The assertion follows from Lemma 3.1 and Lemma 4.1 when $\lkd{X-B} \leq 0$ because $B$ is then an SNC-divisor and so $\lkd{X-B} = \kappa(X, K_X+B)$.
If $\lkd{X-B} = 2$, then $2 = \lkd{X-B} \leq \kappa(X, K_X+B) \leq 2$ and so $\lkd{X-B} = \kappa(X, K_X+B) = 2$. This proves (1) of Theorem 1.1 when $h^1(\SO_X) \geq 2$.
\medskip

\noindent
\textbf{Case: $h^1(\SO_X) = 1$.}
We prove (1) of Theorem 1.1 and complete the proof of (3) of Theorem 1.1. We prove the following lemma.

\begin{lem}
With the same notations and assumptions as in Lemma {\rm 5.1}, assume further that the fibration $\rho$ is an elliptic fibration. Then $B$ is an SNC-divisor.
\end{lem}

\begin{proof}
By (1) of Lemma 5.1, $C = \sum_{i=1}^r F_i$ is a sum of $r$ fibers of $\pi$, here $F_i$ is an elliptic curve with $F_i^2 = 0$ for $i = 1, \ldots, r$. Let $\tilde{F}_i$ (resp.\ $F'_i$) be the proper transform of $F_i$ on $\tilde{X}$ (resp.\ $X$). Then $\tilde{B} - \sum_{i=1}^r \tilde{F}_i$ (resp.\ $B - \sum_{i=1}^r F'_i$) consists only of smooth rational curves, which are fiber components of the $\BP^1$-fibration $\varphi$ on $X$ (resp.\ the $\BP^1$-fibration $\varphi \circ \mu$ on $\tilde{X}$). 
So $B - \sum_{i=1}^r F'_i$ is an SNC-divisor on $X$. 

Suppose to the contrary that $B$ is not an SNC-divisor. 
Then, by the minimality of $\mu: \tilde{X} \to X$, the last exceptional curve, say $\tilde{E}$, in the process of $\mu$ is an irreducible component of $\tilde{B}$ and $\tilde{E} \cdot (\tilde{B}-\tilde{E}) \geq 2$. 
Further, $\tilde{E} \cdot (\sum_{i=1}^r \tilde{F}_i) \geq 1$. 
The image of $\tilde{B}$ by the blowing-down of $\tilde{E}$ is not an SNC-divisor. 
So we see that $f_*(\tilde{E}) \not= 0$.
Since $f_*(\tilde{E})^2 \geq -1$, the coefficient of $f_*(\tilde{E})$ in $D^{\#}$ equals one. Since $C = g_*(D^{\#}) = \sum_{i=1}^r F_i$ is a disjoint union of elliptic curves, $g_*(f_*(\tilde{E})) = 0$. In particular, $f_*(\tilde{E})$ is a $(-1)$-curve. So $f$ is isomorphic on a Zariski open subset containing $\tilde{E}$. 
Let $D^{(1)}$ be the connected component of $D$ containing $f_*(\tilde{E})$. Since $f_*(\tilde{E})$ is a $(-1)$-curve, we infer from Lemma 2.1 that every irreducible component of $D^{(1)}$ has positive coefficient in $D^{\#}$. Then, the image of $D^{(1)}$ by the blowing-down of $f_*(\tilde{E})$ is not an SNC-divisor. So $C = g_*(D^{\#})$ is not an SNC-divisor, which is a contradiction. 
This proves Lemma 5.3.
\end{proof}

We prove (1) of Theorem 1.1 in this case. If $\lkd{X-B} = 2$, then $\lkd{X-B} = \kappa(X, K_X+B) = 2$. If $\lkd{X-B} \leq 0$, then $B$ is an SNC-divisor by Lemma 3.1 and Proposition 4.3. So, $\lkd{X-B} = \kappa(X, K_X + B)$. We consider the case $\lkd{X-B} = 1$ and use the notations in Lemma 5.1. If the fibration $\rho$ is a $\BP^1$-fibration, then by using the same argument as in the case $h^1(\SO_X) \geq 2$, we see that $\lkd{X-B} = \kappa(X, K_X+B) = 1$. If $\rho$ is an elliptic fibration, then $B$ is an SNC-divisor by Lemma 5.3, and hence $\lkd{X-B} = \kappa(X, K_X+B)$. This proves (1) of Theorem 1.1.

We complete the proof of (3) of Theorem 1.1. The assertion follows from the results of Sections 3 and 4 when $\lkd{X-B} \leq 0$. So we assume that $\lkd{X-B} = 1$ and use the notations in \S \S 5.1.
If the fibration $\rho$ on $V$ is a $\BP^1$-fibration, $\tilde{F} \cdot D = 2$ for a fiber $\tilde{F}$ of $\rho$. This $\BP^1$-fibration $\rho$ induces a $\BP^1$-fibration, say $\varphi$, on $X$, i.e., $\varphi = \rho \circ f \circ \mu^{-1}: X \to \Delta$ and $g(\Delta) = h^1(\SO_X) = 1$. Then $F \cdot B = 2$ for a fiber $F$ of $\varphi$. By the same argument as in the proof of (2) of Theorem 1.1 in \S \S 5.2, we obtain a curve $B'$ such that $B \subset B'$, $\lkd{X-B'} = \lkd{X-B} = 1$ and $B'$ satisfies the conditions in (3) of Theorem 1.1. In fact we have a birational morphism $g: X \to \Sigma$ onto a $\BP^1$-bundle $\Sigma$ over $\Delta$ such that $g(B')$ is one of (iv) and (v) in (3) of Theorem 1.1. 

Suppose that $\rho$ is an elliptic fibration. Then Lemma 5.4 implies that $B$ is an SNC-divisor, namely, $\mu = \id$. 
Let $E_1, \ldots, E_n$ be the exceptional curves for $g:= h \circ f$ not contained in $B$ and set $B' := B + E_1+ \cdots + E_n$. 
We know that $\Sigma:= g(X)$ is a $\BP^1$-bundle over $T$ and $g(B') = g(B)$ is (iv) in (3) of Theorem 1.1. 
Since $g(B') = C$, $\lkd{X-B'} = \lkd{X-B} = 1$.
\medskip

The proof of  Theorem 1.1 is thus completed.

\subsection{Proof of Theorem 1.2}

In this subsection, we prove Theorem 1.2. We use the same notations in \S \S 5.1, here we assume further that $B$ is an SNC-divisor. So $\mu = \id$, $\tilde{X} = X$, and $\tilde{B} = B$. The fibration $\pi$ in Lemma 5.1 is either an elliptic fibration or a $\BP^1$-fibation. By Lemma 2.3 and by the construction of $(V,D)$, we know that $(K_X+B)^{+} = f^*(K_V + D^{\#})$. Further, $(K_X+B)^{+} = (h \circ f)^*(K_W + C)$. 
\medskip

\noindent
{\bf Part I: $\pi$ is a $\BP^1$-fibration.} We prove Theorem 1.2 when $\pi$ is a $\BP^1$-fibration. The argument in this case is almost similar to Cases 1$\sim$4 in \cite[Section 3]{K25}. For the reader's convenience, we reproduce the argument. 

We use the notations in (2) of Lemma 5.1. For a positive integer $m$, $|\lfloor m(K_X+B)^+ \rfloor|$ induces a $\BP^1$-fibration on $X$ if and only if so does $| \lfloor m(K_{W} + C) \rfloor|$ on $W$ because $(K_X+B)^{+} = (h \circ f)^*(K_W + C)$.  

By $\lkd{X-B} = \lkd{V-D} = 1$, we have
$$
\deg (K_{\Delta} + \delta) + \sum_{i=1}^j d_i > 0. \eqno{(5.3)}
$$
For a positive integer $m$, we set 
$$
\delta_m := m(K_{\Delta} +\delta)+ \sum_{i=1}^j \lfloor md_i \rfloor \pi(F_i).
$$
Here $\lfloor r \rfloor$ means the integral part of a real number $r$. 
We need to find the least integer $M$ such that, for any $m \geq M$, 
$$
\deg \delta_m = m (2g(\Delta) -2 + t) + \sum_{i=1}^j \lfloor md_i \rfloor \geq 2g(\Delta) + 1 \eqno{(5.4)}
$$
holds.
We consider the following cases separately.
\medskip

\noindent
\textbf{Case 1: }$t\geq 3$. For any $m \geq 1$, $\deg \delta_m \geq 2m g(\Delta) + m(t-2) \geq m (2g(\Delta) + 1)$. So, if $m \geq 1$, (5.4) holds.
\medskip

\noindent
\textbf{Case 2:} $g (\Delta) \geq 2$ (and $t \leq 2$). We consider the following subcases separately.
\smallskip

\noindent
\textbf{2-1:} $t \geq 0$. Then $\deg \delta_m \geq m (2 g(\Delta) -2 ) = 2m (g(\Delta)  - 1)$. So, if $m \geq 3$, (5.4) holds.
\smallskip

\noindent
\textbf{2-2:} $t < 0$. In this subcase, $\ch(k) = 2$. By (2) of Lemma 5.1, $H$ is irreducible, $\pi|_{H}: H \to \Delta$ is a purely inseparable double covering  and $t = 1 - g(\Delta)$.
Then $\deg \delta_m \geq m (g(\Delta)-1) $. So, if $\displaystyle m \geq 2 + \frac{3}{g(\Delta) - 1}$, (5.4) holds. 
\medskip

\noindent
\textbf{Case 3:} $g(\Delta) = 1$ and $1 \leq t \leq 2$. We consider the following subcases separately.
\smallskip

\noindent
\textbf{3-1:} $t = 2$. Then $\deg \delta_m \geq 2m$. So, if $m \geq 2$, (5.4) holds.
\smallskip

\noindent
\textbf{3-2:} $t = 1$. Then $\deg \delta_m \geq m$. So, if $m \geq 3$, (5.4) holds.
\medskip

\noindent
\textbf{Case 4:} $g(\Delta) = 1$ and $t \leq 0$. Then $t = 0$ by (2) of Lemma 5.1. By (5.3), we have $j > 0$. We consider the following subcases separately. 
\smallskip

\noindent
\textbf{4-1:} $H = H_1 + H_2$, where $H_1$ and $H_2$ are sections of $\pi$. 
Since $t = 0$, $H_1 \cap H_2 = \emptyset$. 
By (2) of Lemma 5.1, $\displaystyle d_i = 1 - \frac{1}{m_i} \geq \frac{1}{2}$ for $i = 1, \ldots, j$, where $m_i \geq 2$ or $m_i = +\infty$. Then we have
$$
\deg \delta_m = \sum_{i=1}^j \lfloor md_i \rfloor  \geq \lfloor md_1 \rfloor \geq \lfloor \frac{m}{2} \rfloor.
$$
So, if $m \geq 6$, (5.4) holds.
\smallskip

\noindent
\textbf{4-2:} $H$ is irreducible and $\pi|_{H}$ is separable. By (2) of Lemma 5.1, $H$ is smooth and $\pi|_{H}$ is an unramified double covering. Further, $\displaystyle d_i = 1 - \frac{1}{m_i} \geq \frac{1}{2}$ for $i = 1, \ldots, j$, where $m_i \geq 2$ or $m_i = +\infty$. 
Then we have
$$
\deg \delta_m = \sum_{i=1}^j \lfloor md_i \rfloor  \geq \lfloor md_1 \rfloor \geq \lfloor \frac{m}{2} \rfloor.
$$
So, if $m \geq 6$, (5.4) holds. 
\smallskip

\noindent
\textbf{4-3:} $H$ is irreducible and $\pi|_{H}$ is not separable. In this subcase, $\ch (k) = 2$ and $\pi|_{H}$ is a purely inseparable double covering. So $\# F \cap H = 1$ for every fiber $F$ of $\pi$. By (2) of Lemma 5.2, $\displaystyle d_i = \frac12 \left( 1- \frac{1}{m_i} \right) \geq \frac{1}{4}$ for $i = 1, \ldots, j$, where $m_i \geq 2$ or $m_i = +\infty$. Then we have
$$
\deg \delta_m = \sum_{i=1}^j \lfloor md_i \rfloor  \geq \lfloor md_1 \rfloor \geq \lfloor \frac{m}{4} \rfloor.
$$
So, if $m \geq 12$, (5.4) holds. 
\smallskip

Therefore, in Part I, $M = 12$. 
\medskip

\noindent
\textbf{Part II: $\pi$ is an elliptic fibration.} We prove Theorem 1.2 when $\pi$ is an elliptic fibration. 
We use the notations in (1) of Lemma 5.1. For a positive integer $m$, $|\lfloor m(K_X+B)^+ \rfloor|$ induces an elliptic fibration on $X$ if and only if so does $| \lfloor m(K_{W} + C) \rfloor|$ on $W$ because $(K_X+B)^{+} = (h \circ f)^*(K_W + C)$.  

Since $\pi$ is relatively minimal, $W$ is a relatively minimal elliptic surface. By (1) of Lemma 5.1, $\Delta = \BP^1$, $t = \deg \delta = \chi(\SO_W) + {\rm length} \ST = {\rm length} \ST$.
Furthermore, we have 
$$
K_W + C = \pi^*((t-2)P) + \sum_{r = 1}^s a_{r} E_{r} + \sum_{i=1}^j F_i,
$$
where $P$ is a point of $\Delta = \BP^1$, $j, s \geq 0$, $a_{r} E_{r}$ ranges over all multiple fibers of $\pi$ with multiplicity $m_{r}$, $0 \leq a_{r} < m_{r}$, and $a_{r} = m_{r} - 1$ if $m_{r} E_{r}$ is not a wild fiber of $\pi$. Since $W$ is ruled, $K_W$ is not pseudo effective. This implies that $t = 0,1$ and $C = \sum_{i=1}^j F_i > 0$, i.e., $j \geq 1$. 
Since $\lkd{X-B} = 1$, $(K_W + C) \cdot A > 0$ for any ample divisor $A$ on $W$. So we have
$$
t-2 + \sum_{r = 1}^s \frac{a_{r}}{m_{r}} + \sum_{i=1}^j \frac{1}{n_i} > 0, \eqno{(5.5)}
$$
where $n_i F_i$ is the scheme-theoretic fiber of $\pi$ containing $F_i$. 

For a positive integer $m$, we set 
$$
\delta_m := m(t-2)P + \sum_{r = 1}^s \lfloor \frac{m a_{r}}{m_{r}} \rfloor  \pi(E_{r}) + \sum_{i=1}^j \lfloor \frac{m}{n_i} \rfloor \pi(F_i).
$$
Since $\Delta = \BP^1$,  we need to find the least integer $M$ such that, for any $m \geq M$, 
$$
\deg \delta_m =  m(t-2) + \sum_{r= 1}^s \lfloor \frac{m a_{r}}{m_{r}} \rfloor  + \sum_{i=1}^j \lfloor \frac{m}{n_i} \rfloor \geq 1 (= 2g(\Delta) + 1) \eqno{(5.6)}
$$
holds.
We consider the following cases separately.
\medskip

\noindent
{\bf Case 1: $t=0$.} In this case, $a_{r} = m_{r} - 1$ for $r = 1, \ldots, s$. By (5.5), we have
$$
\sum_{r = 1}^s \frac{m_r-1}{m_r}  + \sum_{i=1}^j \frac{1}{n_i} > 2. \eqno{(5.7)}
$$
We consider the following subcases separately.
\medskip 

\noindent
\textbf{1-1:} $s=0$. Note that $W \cong \BP^1 \times T$ and $n_i = 1$ for $i=1, \ldots, j$. Then
$$
K_W+C = -2\pi^*(P) + \sum_{i=1}^j F_i = \pi^*((j-2)P)
$$
and so $\deg \delta_m = j - 2$. 
By (5.7), $j \geq 3$. So, if $m \geq 1$, (5.6) holds. 
\smallskip

\noindent
\textbf{1-2:} $s=1$. By \cite[Corollary 4.2]{KU85}, the unique multiple fiber $m_1 E_1$ must be a wild fiber. This contradicts $t = 0$. Hence, this subcase does not take place. 
\smallskip

\noindent
\textbf{1-3:} $s=2$. By (5.7), $\displaystyle \sum_{i=1}^j \frac{1}{n_i} > \frac{1}{m_1} + \frac{1}{m_2}$. Since $n_i \in \{ 1, m_1, m_2 \}$ for $i = 1, \ldots, j$, we may assume that $n_1 = 1$. Then $K_W + C = -\pi^*(P) + (m_1-1) E_1 + (m_2-1) E_2 + \sum_{i=2}^j F_j$ and so
\begin{eqnarray*}
\delta_m & = & -mP + \lfloor m\left( 1 - \frac{1}{m_1}\right) \rfloor  \pi(E_1) + \lfloor m\left(1 - \frac{1}{m_2}\right) \rfloor  \pi(E_2)+ \sum_{i=2}^j \lfloor \frac{m}{n_i} \rfloor \pi(F_i) \\
  & \sim &  \left(-m +  \lfloor m\left( 1 - \frac{1}{m_1}\right) \rfloor  + \lfloor m\left( 1 - \frac{1}{m_2}\right) \rfloor +  \sum_{i=2}^j \lfloor \frac{m}{n_i} \rfloor \right) P.
\end{eqnarray*}
\smallskip

\noindent
\textbf{1-3-1}. Suppose that $j = 1$. We may assume that $m_1 \geq m_2$. By (5.7), $m_2 \geq 3$. Then 
\begin{eqnarray*}
\deg \delta_m & \geq & -m + \lfloor m\left( 1 - \frac{1}{m_1}\right) \rfloor  + \lfloor m\left( 1 - \frac{1}{m_2}\right) \rfloor  \\
  & \geq &  -m + \lfloor \frac{m}{2} \rfloor  + \lfloor \frac{2m}{3} \rfloor.
\end{eqnarray*}
So, if $m \geq 8$,  (5.6) holds.
\smallskip

\noindent
\textbf{1-3-2}. Suppose that $j \geq 2$. Note that $n_2 \in \{ 1, m_1, m_2 \}$. If $n_2 = 1$, then 
$$
\deg \delta_m \geq \lfloor m\left(1 - \frac{1}{m_1}\right) \rfloor  + \lfloor m\left(1 - \frac{1}{m_2}\right) \rfloor \geq  2  \lfloor \frac{m}{2} \rfloor.
$$
So, if $m \geq 2$, (5.6) holds.

If $n_2 \not= 1$, then we may assume $m_2 = n_2$. Then $K_W + C = (m_1-1) E_1 + \sum_{i=3}^j F_i$. Hence,
$$
\deg \delta_m \geq \lfloor m\left( 1 - \frac{1}{m_1}\right) \rfloor  \geq \lfloor \frac{m}{2} \rfloor.
$$
So, if $m \geq 2$, (5.6) holds.
\smallskip

Therefore, in 1-3, $M = 8$.
\smallskip

\noindent
\textbf{1-4:} $s \geq 3$. Since $W$ is ruled, $s = 3$ and $j \geq 1$. Then $K_W + C = -2\pi^*(P) + (m_1-1) E_1 + (m_2-1) E_2 + (m_3-1) E_3 + \sum_{i=1}^j F_j$. 
\smallskip

\noindent
\textbf{1-4-1}. We assume that $\exists i \in \{ 1, \ldots, j \}$ s.t. $n_i = 1$. Then we may assume that $n_1 = 1$. So $K_W + C = -\pi^*(P) + (m_1-1) E_1 + (m_2-1) E_2 + (m_3-1) E_3 + \sum_{i=2}^j F_j$. We have
\begin{eqnarray*}
\deg \delta_m & \geq & -m + \lfloor m\left( 1 - \frac{1}{m_1}\right) \rfloor  + \lfloor m\left( 1 - \frac{1}{m_2}\right) \rfloor +  \lfloor m\left( 1 - \frac{1}{m_3}\right) \rfloor  \\
  & \geq &  -m + 3 \lfloor \frac{m}{2} \rfloor.
\end{eqnarray*}
So, if $m \geq 4$, (5.6) holds. 
\smallskip

\noindent
\textbf{1-4-2}. We assume that $n_i > 1$ for $i = 1, \ldots, j$. Then we may assume further that $F_1 = E_1$. Then $K_W + C = -\pi^*(P) + (m_2-1) E_2 + (m_3-1) E_3 + \sum_{i=2}^j F_j$. By (5.7),
$$
\frac{m_2-1}{m_2} + \frac{m_3-1}{m_3}   + \sum_{i=2}^j \frac{1}{n_i} > 1.
$$

If $j=1$, then $\displaystyle 1 > \frac{1}{m_2} + \frac{1}{m_3}$. So, we may assume that $m_3 \geq 3$. Then we have
\begin{eqnarray*}
\deg \delta_m & \geq & -m + \lfloor m\left( 1 - \frac{1}{m_2}\right) \rfloor +  \lfloor m\left( 1 - \frac{1}{m_3}\right) \rfloor  \\
  & \geq &  -m + \lfloor \frac{m}{2} \rfloor +  \lfloor \frac{2m}{3} \rfloor.
\end{eqnarray*}
So, if $m \geq 8$, (5.6) holds.

If $j \geq 2$, then we may assume further that $n_2 = m_2$. Then
$$
\deg \delta_m \geq \lfloor m\left( 1 - \frac{1}{m_3}\right) \rfloor  \geq  \lfloor \frac{m}{2} \rfloor.
$$
So, if $m \geq 2$, (5.6) holds.
\smallskip

Therefore, in 1-4, $M = 8$. 
\medskip

\noindent
\textbf{Case 2: $t=1$.} We may assume that $m_1 E_1$ is the wild fiber of $\pi$. Then $a_r = m_r - 1$ for $r = 2, \ldots, s$. By (5.5), 
$$
\frac{a_1}{m_1} +  \sum_{r = 2}^s \frac{m_r-1}{m_{r}} + \sum_{i=1}^j \frac{1}{n_i} > 1. \eqno{(5.8)}
$$
Set $p = \ch(k)$.  We consider the following subcases separately.
\medskip

\noindent
\textbf{2-1:} $s=1$. Then $K_W+C = -\pi^*(P) +  a_1 E_1 + \sum_{i=1}^j F_j$ and 
$$
\delta_m = -mP + \lfloor \frac{ma_1}{m_1} \rfloor  \pi(E_1) + \sum_{i=1}^j \lfloor \frac{m}{n_i} \rfloor \pi(F_i).
$$
By \cite[Theorem 2]{BM77}, either $a_1 = m_1 - 1$ or $a_1 = m_1 - \nu_1 -1$, where $\nu_1$ is a positive integer such that $\nu_1 \ | \ m_1$. 
\smallskip

\noindent
\textbf{2-1-1.} Suppose that $a_1 = m_1 - 1$. By (5.8), $\displaystyle \sum_{i=1}^j \frac{1}{n_i} > \frac{1}{m_1}$. Hence either $j \geq 2$ or $j=1$ and $n_1 = 1$. 

If $j \geq 3$, then we may assume that $n_2 = n_3 = 1$. Then we have
$$
\deg \delta_m \geq -m +  \lfloor m \left( 1-\frac{1}{m_1} \right) \rfloor + \lfloor \frac{m}{n_1} \rfloor + 2m \geq m \geq 1.
$$
So, if $m \geq 1$, (5.6) holds. 

If $j = 2$, then we may assume that $n_2 = 1$. Then we have
$$
\deg \delta_m =  \lfloor m \left( 1-\frac{1}{m_1} \right) \rfloor + \lfloor \frac{m}{n_1} \rfloor  \geq \lfloor \frac{m}{2} \rfloor + \lfloor \frac{m}{n_1} \rfloor.
$$
So, if $m \geq 2$, (5.6) holds.

If $j = 1$, then $n_1 = 1$ and so
$\displaystyle 
\deg \delta_m \geq  \lfloor \frac{m}{2} \rfloor
$.
So, if $m \geq 2$, (5.6) holds.
\smallskip

Therefore, in 2-1-1, $M = 2$.
\smallskip

\noindent
\textbf{2-1-2.} Suppose that $a_1 = m_1 - \nu_1 - 1$ and $\nu_1 \ | \ m_1$.
By (5.8), $\displaystyle \sum_{i=1}^j \frac{1}{n_i} > \frac{\nu_1 + 1}{m_1}$.
By \cite[Corollary 4.2]{KU85}, $m_1 = p^{\nu}$ with $\nu > 0$ and $\nu_1 =1$. Hence  $\displaystyle \sum_{i=1}^j \frac{1}{n_i} > \frac{2}{m_1}$.

Suppose that $j \geq 2$. Then we may assume further that $n_2 = 1$. Then we have
$$
\deg \delta_m \geq  \lfloor m \left( 1-\frac{2}{m_1} \right) \rfloor + \lfloor \frac{m}{n_1}  \rfloor .
$$
If $m_1 = 2$, then $\displaystyle \deg \delta_m \geq   \lfloor \frac{m}{2}  \rfloor$ since $n_1 \leq 2$. So, if $m \geq 2$, (5.6) holds. If $m_1 \geq 3$, then $\displaystyle 1 - \frac{2}{m_1} \geq \frac13$. So we have
\begin{eqnarray*}
\deg \delta_m & \geq & -m + \lfloor m\left( 1 - \frac{2}{m_1}\right) \rfloor + \lfloor \frac{m}{n_1} \rfloor +  \lfloor \frac{m}{n_2} \rfloor  \\
  & \geq &   \lfloor \frac{m}{3} \rfloor +  \lfloor \frac{m}{n_1} \rfloor.
\end{eqnarray*}
So, if $m \geq 3$, (5.6) holds.

Suppose that $j = 1$. Then $\displaystyle \frac{1}{n_1} > \frac{2}{m_1}$ and so $m_1 \geq 3$ and $n_1=1$.
We have
\begin{eqnarray*}
\deg \delta_m & = & -m + \lfloor m\left( 1 - \frac{2}{m_1}\right) \rfloor + \lfloor \frac{m}{n_1} \rfloor  \\
  & = &  \lfloor m\left( 1 - \frac{2}{m_1}\right) \rfloor \\
  & \geq & \lfloor \frac{m}{3} \rfloor.
\end{eqnarray*}
So, if $m \geq 3$, (5.6) holds.
\smallskip

Therefore, in 2-1-2, $M = 3$.
\medskip

\noindent
\textbf{2-2:} $s=2$. Then $m_1 E_1$ and $m_2 E_2$ exhaust the multiple fibers of $\pi$ and $m_1 E_1$ is the unique wild fiber of $\pi$. 
Then $a_2 = m_2 -1$ and either $a_1 = m_1 - 1$ or $a_1 = m_1 - \nu_1 -1$ and $\nu_1 \ | \ m_1$. If $a_1 = m_1 - 1$, then $K_W = -\pi^*(P) + (m_1 - 1) E_1 + (m_2-1) E_2$. So $2K_W \geq 0$, which is a contradiction because $W$ is ruled. Hence,  $a_1 = m_1 - \nu_1 -1$ and $\nu_1 \ | \ m_1$. By (5.8), we have
$$
1 + \sum_{i=1}^j \frac{1}{n_i} > \frac{\nu_1 + 1}{m_1} + \frac{1}{m_2}.
$$
\smallskip

\noindent
\textbf{2-2-1.} We assume that $n_i = 1$ for some $i \in \{ 1, \ldots, j \}$. We may assume that $n_1 = 1$. Then 
\begin{eqnarray*}
\deg \delta_m & \geq & -m + \lfloor m\left( 1 - \frac{\nu_1 + 1}{m_1}\right) \rfloor + \lfloor m\left( 1 - \frac{1}{m_2}\right) \rfloor +  \lfloor \frac{m}{n_1} \rfloor  \\
  & \geq &  \lfloor m\left( 1 - \frac{1}{m_2}\right) \rfloor  \geq  \lfloor \frac{m}{2} \rfloor.
\end{eqnarray*}
So, if $m \geq 2$, (5.6) holds.
\smallskip

\noindent
\textbf{2-2-2.} We assume that $n_i \geq 2$ for $i = 1, \ldots, j$ and $j = 2$. Then we may assume that $F_1 = E_1$ and $F_2 = E_2$. Then $K_W + C = -\pi^*(P) + a_1 E_1 + (m_2-1) E_2 + E_1 + E_2 = (a_1+1)E_1 = (m_1-\nu_1)E_1$. Since $\nu_1 \ | \ m_1$ and $\lkd{X-B} = 1$, $m_1 \geq 2 \nu_1$. Then we have 
$$
\deg \delta_m = \lfloor m \left( 1- \frac{\nu_1}{m_1} \right) \rfloor \geq \lfloor m \left( 1- \frac{\nu_1}{2 \nu_1} \right) \rfloor  \geq \lfloor \frac{m}{2} \rfloor.
$$
So, if $m \geq 2$, (5.6) holds.
\smallskip

\noindent
\textbf{2-2-3:} Assume that $j=1$ and $n_1 \geq 2$. Then $F_1 = E_1$ or $E_2$. In 2-2-3, we assume that $F_1 = E_2$. Then $K_W + C = -\pi^*(P) + a_1 E_1 + (m_2-1) E_2 + E_2 = a_1 E_1$, where $a_1 = m_1 - \nu_1 - 1$ and $\nu_1 \ | \ m_1$. Since $\lkd{X-B} = 1$, $a_1 > 0$. By \cite[(1.6)]{KU85}, $m_1 = p^{\gamma} \nu_1$ for some $\gamma > 0$. 
We have
$$
\deg \delta_m = \lfloor m \left( 1 - \frac{\nu_1+ 1}{m_1}\right) \rfloor =  \lfloor m\left(1 - \frac{\nu_1+ 1}{p^{\gamma} \nu_1}\right) \rfloor =  \lfloor m \left( 1 - \frac{1}{p^{\gamma}} - \frac{1}{p^{\gamma} \nu_1} \right) \rfloor.
$$

If $p \geq 3$, then $\displaystyle 1 - \frac{1}{p^{\gamma}} - \frac{1}{p^{\gamma} \nu_1}  \geq \frac13$. So, if $m \geq 3$, (5.6) holds.

If $p = 2$, then, by $a_1 = m_1 - \nu_1 - 1 >0$, $m_1 \geq 4$. We have
$$
\deg \delta_m  = \lfloor m \left( 1 - \frac{1}{p^{\gamma}} - \frac{1}{p^{\gamma} \nu_1} \right) \rfloor \geq \lfloor \frac{m}{4} \rfloor.
$$
So, if $m \geq 4$, (5.6) holds. 

Therefore, in 2-2-3, $M = 4$.
\smallskip

\noindent
\textbf{2-2-4.} We assume that $j=1$, $n_1 \geq 2$ and $F_1 = E_1$. Then we have $K_W + C = - \pi^*(P) + (a_1 + 1) E_1 + (m_2 - 1)E_2$. By (5.8), $\displaystyle \frac{a_1+1}{m_1} + \frac{m_2-1}{m_2} > 1$. So,
$$
\frac{a_1+1}{m_1} > \frac{1}{m_2}.  \eqno{(5.9)}
$$
Since $a_1 = m_1 - \nu_1 -1$ and $\nu_1 \ | \ m_1$, it follows from \cite[(1.6) and (1.7)]{KU85} that $m_1 = p^{\gamma} \nu_1$ for some $\gamma > 0$. By \cite[Theorm 3.3]{KU85}, $(m_1, m_2 \ | \ \nu_1, m_2)$ satisfies the conditions $U_1$ and $U_2$ in \cite[Definition 3.2]{KU85}. By the condition $U_2$, we have $m_2 \ | \ m_1$. If $p \ | \ \nu_1$, then $m_1 \ | \ m_2$ by the condition $U_1$. If $p  \not| \ \nu_1$, then we have $m_2 = p^{\beta} \nu_1$ with a non-negative integer $\beta \leq \gamma$ by the condition $U_1$. Hence we get the following two cases.
\begin{enumerate}
\item[(i)]
$p \ | \ \nu_1$, $m_1 = m_2 = p^{\gamma} \nu_1$ and $\gamma \geq 1$.
\item[(ii)]
$p  \not| \ \nu_1$, $m_1 = p^{\gamma} \nu_1$,  $m_2 = p^{\beta} \nu_1$, $\gamma \geq 1$ and $\gamma \geq \beta$. 
\end{enumerate}
The above argument is the same as in \cite[p.\ 306]{KU85}. 
\smallskip

\noindent
Case (i). The we have
$$
\deg \delta_m  = - m + \lfloor m \left( 1 - \frac{1}{p^{\gamma}}\right) \rfloor + \lfloor m \left( 1 - \frac{1}{p^{\gamma}\nu_1}\right) \rfloor
$$
If $p \geq 3$, then $\nu_1 \geq 3$ and so $m_1 = m_2 \geq 9$. We have
$$
\deg \delta_m  \geq  - m + \lfloor \frac{2m}{3} \rfloor + \lfloor \frac{8m}{9} \rfloor
$$
So, if $m \geq 3$, (5.6) holds.
If $p = 2$, then $\nu_1 \geq 2$ and $m_1 = m_2 \geq 4$. We have
$$
\deg \delta_m  \geq  - m + \lfloor \frac{m}{2} \rfloor + \lfloor \frac{3m}{4} \rfloor.
$$
So, if $m \geq 6$, (5.6) holds.

In Case (i), $M = 6$.
\smallskip

\noindent
Case (ii). Then we have
$$
\deg \delta_m  = - m + \lfloor m \left( 1 - \frac{1}{p^{\gamma}}\right) \rfloor + \lfloor m \left( 1 - \frac{1}{p^{\beta}\nu_1} \right) \rfloor
$$
and $p \not| \ \nu_1$. 
If $p \geq 3$, then $p^{\beta} \nu_1 \geq 2$ and $p^{\gamma} \geq 3$. We have
$$
\deg \delta_m \geq  -m + \lfloor \frac{2m}{3} \rfloor + \lfloor \frac{m}{2} \rfloor.
$$
So, if $m \geq 8$, (5.6) holds.

Suppose that $p = 2$. If $\beta = 0$, then $\nu_1 \geq 3$. We have
$$
\deg \delta_m \geq  -m + \lfloor \frac{m}{2} \rfloor + \lfloor \frac{2m}{3} \rfloor.
$$
So, if $m \geq 8$, (5.6) holds.
If $\beta \geq 1$ and $\gamma \geq 2$, then $p^{\gamma} \geq 4$ and $p^{\beta} \nu_1 \geq 2$. We have
$$
\deg \delta_m \geq  -m + \lfloor \frac{3m}{4} \rfloor + \lfloor \frac{m}{2} \rfloor.
$$
So, if $m \geq 6$,  (5.6) holds.
If $\gamma = 1$ and $\beta > 0$, then $\gamma = \beta = 1$ and $m_1 = m_2 = 2\nu_1$ and $2 \not| \ \nu_1$. 
By (5.9), $\displaystyle \frac{\nu_1}{2\nu_1} > \frac{1}{2 \nu_1}$. By $2 \not| \nu_1$, $\nu_1 \geq 3$. 
So $m_1 = m_2 = 2 \nu_1 \geq 6$. We have
$$
\deg \delta_m \geq  -m + \lfloor \frac{m}{2} \rfloor + \lfloor \frac{5m}{6} \rfloor.
$$
So, if $m \geq 4$, (5.6) holds.

Therefore, in Case (ii), $M = 8$. 
\smallskip

Therefore, in 2-2-4, $M = 8$. The argument of 2-2 is thus completed.
\medskip

\noindent
\textbf{2-3:} $s \geq 3$. We easily see that this case does not take place since $W$ is ruled.
\smallskip

Therefore, in Part II, $M = 8$. 
\medskip

The proof of Theorem 1.2 is thus completed.


\end{document}